\documentclass[a4paper, 12pt]{article}
\pdfoutput=1 

\usepackage[utf8]{inputenc}
\usepackage[T1]{fontenc}
\usepackage[a4paper, margin=2.5cm]{geometry}
\usepackage{needspace}
\usepackage{libertine} 
\usepackage{inconsolata} 

\usepackage{amsmath, amsthm, amssymb,mathtools}
\usepackage{graphicx}
\usepackage{enumerate}
\usepackage[shortlabels]{enumitem}
\usepackage{authblk}
\usepackage{thmtools}
\usepackage{thm-restate}
\usepackage{authblk}
\usepackage[hyphens]{xurl}
\usepackage[colorlinks=true, linkcolor = bordeaux, citecolor = darkblue, urlcolor = darkblue]{hyperref}
\usepackage{float}
\usepackage[normalem]{ulem}
\usepackage{xcolor}
\usepackage[ruled,vlined]{algorithm2e}
\usepackage{textpos}
\usepackage{soul}

\usepackage{comment}
\usepackage{placeins}
\usepackage{pdfpages}

\usepackage{pgf}
\usepackage{pgffor}
\usepackage{xspace}
\usepackage{tikz}
\usetikzlibrary{fit}
\usetikzlibrary{arrows}
\usetikzlibrary{patterns}
\usetikzlibrary{calc}
\usetikzlibrary{shapes}
\usetikzlibrary{positioning}
\usetikzlibrary{math}
\usetikzlibrary{backgrounds}

\pgfdeclarelayer{background}
\pgfdeclarelayer{foreground}
\pgfsetlayers{background,main,foreground}

\renewenvironment{abstract}
{\small\vspace{-1em}
\begin{center}
\bfseries\abstractname\vspace{-.5em}\vspace{0pt}
\end{center}
\list{}{
\setlength{\leftmargin}{0.6in}%
\setlength{\rightmargin}{\leftmargin}}%
\item\relax}
{\endlist}

\declaretheorem[name=Theorem]{theorem}

\declaretheorem[name=Lemma, numberwithin = section]{lemma}

\declaretheorem[name=Definition, sibling=lemma]{definition}

\declaretheorem[name=Conjecture, sibling=lemma]{conjecture}

\declaretheorem[name=Claim, numberwithin=lemma]{claim}

\declaretheorem[name=Question, sibling=theorem]{question}

\def\cqedsymbol{\ifmmode$\lrcorner$\else{\unskip\nobreak\hfil
\penalty50\hskip1em\null\nobreak\hfil$\lrcorner$
\parfillskip=0pt\finalhyphendemerits=0\endgraf}\fi}

\newcommand{\say}[1]{``#1''} 
\def\B{\mathcal{B}} %
\def\T{\mathcal{T}} 

\makeatletter
\newcommand{\leqnomode}{\tagsleft@true}

\newcommand{\reqnomode}{\tagsleft@false}
\makeatother

\definecolor{CornflowerBlue}{rgb}{0.39, 0.58, 0.93}
\definecolor{DarkGoldenrod}{rgb}{0.72, 0.53, 0.04}
\definecolor{BritishRacingGreen}{rgb}{0.0, 0.26, 0.15}
\definecolor{DarkMagenta}{rgb}{0.55, 0.0, 0.55}
\definecolor{AO}{rgb}{0.0, 0.5, 0.0}
\definecolor{BostonUniversityRed}{rgb}{0.8, 0.0, 0.0}
\definecolor{myRed}{rgb}{0.8, 0.0, 0.0}
\definecolor{DarkMidnightBlue}{rgb}{0.0, 0.2, 0.4}
\definecolor{DarkTangerine}{rgb}{1.0, 0.66, 0.07}
\definecolor{AppleGreen}{rgb}{0.55, 0.71, 0.0}
\definecolor{BrightUbe}{rgb}{0.82, 0.62, 0.91}
\definecolor{Amethyst}{rgb}{0.6, 0.4, 0.8}
\definecolor{DarkGray}{rgb}{0.52, 0.52, 0.51}
\definecolor{Gray}{rgb}{0.66, 0.66, 0.66}
\definecolor{BananaYellow}{rgb}{1.0, 0.88, 0.21}
\definecolor{Amber}{rgb}{1.0, 0.75, 0.0}
\definecolor{LightGray}{rgb}{0.83, 0.83, 0.83}
\definecolor{PrOrange}{rgb}{1.0, 0.56, 0.0}
\definecolor{DeepCarrotOrange}{rgb}{0.91, 0.41, 0.17}
\definecolor{cobalt}{RGB}{0,71,171}
\definecolor{brightpink}{RGB}{255,0,204} 
\definecolor{bordeaux}{RGB}{100,0,50}
\definecolor{olivegreen}{RGB}{107, 142, 35}
\definecolor{forestgreen}{RGB}{34, 139, 34}
\definecolor{darkgreen}{RGB}{0, 100, 0}
\definecolor{darkblue}{RGB}{25, 25, 112}
\definecolor{grun}{rgb}{0, 0.5, 0.5}
\definecolor{violet}{RGB}{177, 0.5, 255}

\definecolor{meikeColour}{rgb}{0.41, 0.16, 0.38}

\definecolor{pabloColor}{rgb}{0.30,0.65,1.00}

\DeclareMathOperator{\tw}{tw}
\DeclareMathOperator{\pw}{pw}
\renewcommand{\leq}{\leqslant}
\renewcommand{\geq}{\geqslant}

\renewcommand{\ge}{\geqslant}
\renewcommand{\emptyset}{\varnothing}
\DeclarePairedDelimiter{\abs}{\lvert}{\rvert}
\DeclarePairedDelimiter{\set}{\{}{\}}
\usepackage[noabbrev,capitalise,nameinlink]{cleveref}
\crefname{claim}{Claim}{Claims}
\crefname{lemma}{Lemma}{Lemmas}
\crefname{theorem}{Theorem}{Theorems}
\crefname{proposition}{Proposition}{Propositions}
\crefname{question}{Question}{Questions}
\crefname{definition}{Definition}{Definitions}
\crefname{conjecture}{Conjecture}{Conjectures}
\crefname{corollary}{Corollary}{Corollaries}
\crefformat{equation}{(#2#1#3)}
\Crefformat{equation}{Equation #2(#1)#3}
\setlist[itemize]{topsep=0ex,itemsep=0ex,parsep=0.25ex}
\setlist[enumerate]{topsep=0ex,itemsep=0ex,parsep=0.25ex}
\definecolor{tikzgray}{gray}{0.6}
\definecolor{tikzblue}{RGB}{50, 50, 224}
\definecolor{tikzred}{RGB}{179, 0, 0}
\tikzset{MyNode/.style={circle, draw, inner sep=2,outer sep=0, fill=tikzgray}}

\let\OLDthebibliography\thebibliography
\renewcommand\thebibliography[1]{
  \OLDthebibliography{#1}
  \setlength{\parskip}{2pt plus 0.3ex minus 0.3ex}
  \setlength{\itemsep}{4pt plus 0.3ex minus 0.3ex}
}

\title{On tree decompositions whose trees are minors}

\author[1]{Pablo Blanco} 
\author[2]{Linda Cook\thanks{Supported by the Institute for Basic Science (IBS-R029-C1).}}
\author[3]{Meike Hatzel\thanks{Supported by the Federal Ministry of Education and
Research (BMBF) and by a fellowship within the IFI programme of the German Academic Exchange Service (DAAD).}}
\author[4]{Claire Hilaire}
\author[5]{Freddie Illingworth\thanks{Supported by EPSRC grant EP/V007327/1.}}
\author[1]{Rose McCarty\thanks{Supported by the National Science Foundation under Grant No.\ DMS-2202961.}}

\affil[1]{Department of Mathematics, Princeton University, United States}
\affil[2]{Discrete Math Group, Institute for Basic Science, Daejeon, Republic of Korea}
\affil[3]{National Institute of Informatics, Tokyo, Japan}
\affil[4]{LaBRI, Université de Bordeaux, Bordeaux, France}
\affil[5]{Mathematical Institute, University of Oxford, United Kingdom}

\date{February 23, 2023}

\begin{document}
\maketitle

\begin{textblock}{20}(-2.3, 6.5)
   \includegraphics[width=80px]{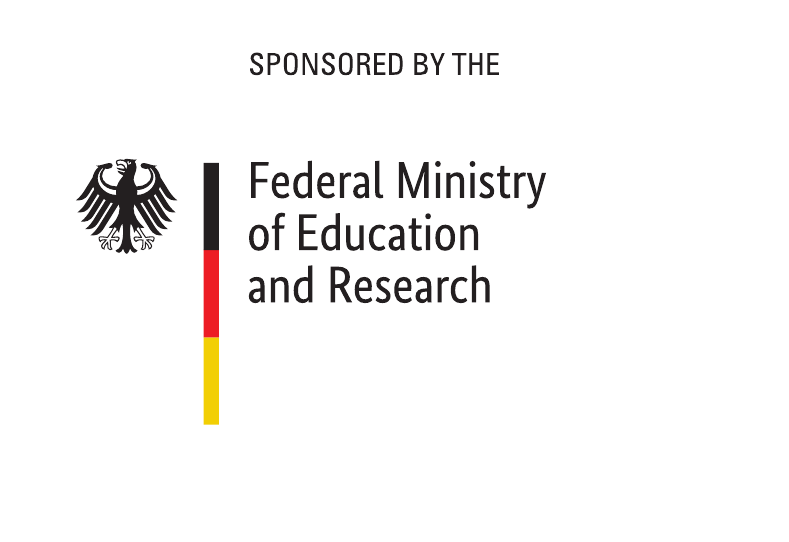}%
\end{textblock}

\begin{abstract}
In 2019, Dvo\v{r}\'{a}k asked whether every connected graph $G$ has a tree decomposition $(T, \B)$ so that $T$ is a subgraph of $G$ and the width of $(T, \B)$ is bounded by a function of the treewidth of $G$.
We prove that this is false, even when $G$ has treewidth $2$ and $T$ is allowed to be a minor of $G$.
\end{abstract}

\section{Introduction}

Suppose that a graph $G$ has small treewidth, and consider all tree decompositions $(T, \B)$ of $G$ whose width is not too much larger than the optimum. To what extent can we choose or manipulate the ``shape'' of $T$?
For graphs with no long path, we can choose $T$ to also have no long path~\cite{sparsity2012}\footnote{The reference gives us an elimination tree $F$ of $G$ of depth $k$. Then we can obtain a tree decomposition $(F, \B)$ of $G$ of width $k$ by letting the bag of each vertex $v \in V(F)$ be the set of all ancestors of $v$ in $F$.}; this gives rise to the parameter called \emph{treedepth}. 
Similarly, for graphs of bounded degree, we can choose $T$ to also have bounded degree~\cite{DingOporowski95}; this relates to the parameters of \emph{congestion} and \emph{dilation}.
Moreover, for graphs excluding any tree as a minor, we can choose $T$ to just be a path~\cite{bienstock1991}; this results in the parameter called \emph{pathwidth}.

It would be wonderful if we could unify all such results into a single theorem which relates the shape of $T$ to $G$. In 2019, Dvo\v{r}\'{a}k suggested one way of accomplishing this goal. In the question below and throughout the paper, we write $\tw(G)$ for the treewidth of~$G$.

\begin{question}[\cite{ZDconjecture}]\label{Dvorak}
    Does there exist a polynomial $P$ such that every connected graph $G$ has a tree decomposition $(T, \B)$ of width at most $P(\tw(G))$ such that $T$ is a subgraph of $G$?
\end{question}

\noindent Unfortunately, we prove that the answer to~\cref{Dvorak} is \say{no} in the following strong sense.

\begin{restatable}{theorem}{mainThm}\label{thm:main}
For every positive integer $k$, there is a connected graph $G$ of treewidth~$2$ such that if $(T, \B)$ is a tree decomposition of $G$ and $T$ is a minor of $G$, then $(T, \B)$ has width at least~$k$.
\end{restatable}

\noindent Intriguingly, in our proof of \cref{thm:main}, it seems crucial that the constructed graphs contain all trees as minors; perhaps \cref{Dvorak} could be true when $\tw(G)$ is replaced by $\pw(G)$, the pathwidth of~$G$. 
In other words, perhaps there exists a polynomial (or even just some function) $P$ so that every connected graph $G$ has a tree decomposition $(T, \B)$ of width at most $P(\pw(G))$ such that $T$ is a subgraph of $G$. 
We leave this as an open problem.

There has been strong interest in obtaining good bounds for treedepth \cite{td2021, tdpw2023, KawarabayashiRossman}, pathwidth \cite{pathwidth21}, and treewidth \cite{chekuriChuzhow2016, bestExcludedGrid} as a function of the natural obstructions (which are paths, trees, and grids, respectively\footnote{Formally, a class of graphs has bounded treedepth/pathwidth/treewidth if and only if it does not contain all paths/trees/grids as minors, respectively. See~\cite{sparsity2012}, \cite{bienstock1991}, and \cite{graphMinors4} for the respective proofs. Note that sometimes the obstructions are considered as subgraphs or subdivisions rather than minors. This occurs when the two definitions are equivalent, for instance when considering paths as minors (or equivalently as subgraphs).}).
These problems were in large part motivated by the desire to obtain better approximation algorithms and better win-win algorithms based on the obstructions.
An affirmative answer to \cref{Dvorak} would have unified these approaches, but unfortunately \cref{thm:main} shows that this is not possible.

There has also been recent interest in finding the $2$-connected obstructions for treedepth \cite{treedepth2Conn} and pathwidth~\cite{pathwidth2ConnDT, pathwidth2Conn} in $2$-connected graphs.
It seems unlikely that requiring $G$ to be $2$-connected would change the answer to \cref{Dvorak}, but the graphs we construct for \cref{thm:main} are not $2$-connected, thus leaving this as an open possibility.

We present a self-contained proof of \cref{thm:main}, however some steps were discovered independently by Hickingbotham~\cite{hickingbothamThesis}. In particular, Hickingbotham~\cite[Lemma 7.2.1]{hickingbothamThesis} noticed that it is just as hard to ensure $T$ is a subgraph of $G$ in \cref{Dvorak} as it is to ensure $T$ is a minor of $G$.
Thus our main contribution is \cref{lem:reduction}, which essentially shows that we can also force each vertex of $T$ to be in its own bag. Hickingbotham~\cite[Theorem~7.5.1]{hickingbothamThesis} already proved that this stronger condition can blow up the width. Moreover, Hickingbotham proved some positive results, including that the answer to \cref{Dvorak} is \say{yes} if $G$ is an outerplanar graph~\cite[Theorem~7.3.3]{hickingbothamThesis}. Note that outerplanar graphs are the graphs with simple treewidth at most $2$~\cite{KU12}, and so in this sense \cref{thm:main} is optimal.

We outline our approach to proving \cref{thm:main} in more detail in the next section.

\section{Preliminaries}
We use the following \say{subtree view} of tree decompositions. Recall that a \emph{subtree} of a graph $G$ is any subgraph of $G$ which is connected and acyclic.

\begin{definition}\label{def:td-subtree-version}
    Let $G$ be a graph, let $T$ be a tree, and let $\B = \set{B_x \colon x \in V(T)}$ be a family of subsets of $V(G)$ indexed by the vertices of $T$.
    For each vertex $v$ of $G$, we define
    \begin{equation*}
        T_v \coloneqq T[\set{x \colon v \in B_x}].
    \end{equation*}
    Then $(T, \B)$ is a \emph{tree decomposition} of $G$ if and only if the following conditions both hold.
    \begin{itemize}
        \item Each $T_v$ is a non-empty subtree of $T$.
        \item If $uv \in E(G)$, then $V(T_u) \cap V(T_v) \neq \emptyset$.
    \end{itemize}
\end{definition}

We use this notation $T_v$ throughout the paper. When there is no chance for confusion, we refer to $T_v$ and its vertex set $V(T_v)$ interchangeably. The \emph{width} of $(T, \B)$ is then the maximum, over all $x \in V(T)$, of $\abs{\set{v \in V(G) \colon x \in T_v}}-1$. The \emph{treewidth} of $G$ is the minimum width of a tree decomposition of $G$. Note that, if we are given a tree $T$ and a collection $(T_v \colon v \in V(G))$ of subtrees of $T$ which satisfy the conditions from \cref{def:td-subtree-version}, then we can define a tree decomposition $(T, \B)$ of $G$ by setting $B_x \coloneqq \set{v \in V(G) \colon x \in T_v}$ for each $x \in V(T)$.

We now outline our overall strategy for proving \cref{thm:main}. This theorem equivalently says that \cref{conj:minor} below is false, even for connected graphs of treewidth~$2$. We disprove \cref{conj:minor} by reducing each of the three conjectures below to the next one, and then disproving the final conjecture. Afterwards, we evaluate the treewidth of the constructed counterexamples more carefully. Note that in \cref{conj:strong}, the condition \say{for every vertex $v$ of $G$, $v \in T_v$} is equivalent to \say{for every vertex $x$ of $T$, $x \in B_x$}.
\begin{conjecture}\label{conj:minor}
    There is a function $f$ such that every connected graph $G$ has a tree decomposition $(T, \B)$ of width at most $f(\tw(G))$ such that $T$ is a minor of $G$.
\end{conjecture}

\begin{conjecture}\label{conj:spanning}
    There is a function $f$ such that every connected graph $G$ has a tree decomposition $(T, \B)$ of width at most $f(\tw(G))$ such that $T$ is a spanning tree of $G$.
\end{conjecture}

\begin{conjecture}\label{conj:strong}
    There is a function $f$ such that every connected graph $G$ has a tree decomposition $(T, \B)$ of width at most $f(\tw(G))$ such that $T$ is a spanning tree of $G$ and, for every vertex $v$ of $G$, we have $v \in T_v$.
\end{conjecture}

Hickingbotham proved that \cref{conj:minor} implies \cref{conj:spanning} in~\cite[Lemma 7.2.1]{hickingbothamThesis}.
In \cref{sec:reduction}, we show that \cref{conj:spanning} implies \cref{conj:strong}; this crucial new step is our main contribution. 
Finally, in \cref{sec:construction}, we construct a graph that does not satisfy \cref{conj:strong}.
Hickingbotham~\cite[Theorem~7.5.1]{hickingbothamThesis} independently discovered a counterexample to \cref{conj:strong} which actually contains our counterexample.
However, we include ours since it is slightly simpler and makes the paper self-contained.
We now conclude this section by providing a short proof that \cref{conj:minor} (where $T$ is a minor) implies \cref{conj:spanning} (where $T$ is a spanning tree), for the sake of completeness.

\begin{lemma}\label{lem:minor-to-subgraph-tree}
If $G$ is a connected graph with a tree decomposition $(T, \B)$ of width $k$ such that
$T$ is a minor of $G$, then there exists a tree decomposition $(T', \B')$ of $G$ of width $k$ such that $T'$ is a spanning tree of~$G$.
\end{lemma}
\begin{proof}
    Since $T$ is a minor of $G$, there exists a collection $(Q_x \colon x \in V(T))$ of pairwise disjoint non-empty subtrees of $G$ such that, for each edge $xy \in E(T)$, there exists an edge $e_{xy} \in E(G)$ with one end in $V(Q_x)$ and the other end in $V(Q_y)$. Since $G$ is connected, we may assume that $V(G) = \cup_{x \in V(T)} V(Q_x)$. Now, let $T'$ be the spanning tree of $G$ which is obtained from $\cup_{x \in V(T)}Q_x$ by adding the edge $e_{xy}$ for all $xy \in E(T)$.

    For each $v \in V(G)$, let $T'_v$ be the subtree of $T'$ which is induced by the union of all sets $V(Q_x)$ such that $x \in T_v$. This collection of subtrees of $T'$ satisfies the conditions of \cref{def:td-subtree-version} and therefore yields a tree decomposition $(T', \B')$. Furthermore, this tree decomposition has the same width at $(T, \B)$, which completes the proof.
\end{proof}

\section{Reduction to \texorpdfstring{\cref{conj:strong}}{Conjecture 2.4}}\label{sec:reduction}

In this section we show that \cref{conj:spanning} (where $T$ is a spanning tree) implies \cref{conj:strong} (where, additionally, for every vertex $v$ of $G$, we have $v \in T_v$).

We use the following well-known fact about tree decompositions of paths. We include a proof for the sake of completeness. The bounds are not optimal; we aim for simplicity instead.

\begin{lemma}\label{lem:longpath}
    For any positive integers $h$ and $k$, if $P$ is a path with at least $(k+2)^h$ vertices and $(T, \B)$ is a tree decomposition of $P$ of width at most $k$, then $T$ contains a path of length~$h$.
\end{lemma}
\begin{proof}
We consider a tree decomposition $(T, \B)$ where $T$ is rooted at an arbitrary vertex ${r \in V(T)}$. The \emph{height} of $T$ is then the maximum length of a path which has $r$ as one of its ends. For fixed $k$, we prove by induction on $h$ that, under the same hypothesis, we actually obtain the following stronger conclusion: that the height of $T$ is at least~$h$. 

The base case of $h=1$ holds since $P$ has more than $k+1$ vertices and therefore $(T, \B)$ has more than one bag. So we may assume that $h>1$ and the claim holds for $h-1$. Observe that we can partition $V(P)$ into $k+2$ sets, each of which induces in $P$ a path with at least $(k+2)^{h-1}$ vertices. Since $(T, \B)$ has width at most $k$, one of these sets is disjoint from the root bag $B_r$. Thus, by the inductive hypothesis, one of the components of $T-\set{r}$, when rooted at its neighbour of $r$, has height at least $h-1$. So $T$ has height at least $h$, as desired.
\end{proof}

We use the following construction to show that \cref{conj:spanning} implies \cref{conj:strong}. Given a positive integer $k$, a graph $G$, and an arbitrary ordering of the vertices of $G$, we define a new graph denoted $\widetilde{G}$. This graph $\widetilde{G}$ is obtained from $G$ by attaching one rooted tree to each vertex of $G$; so $\widetilde{G}$ has the same treewidth as $G$ (unless $E(G) = \emptyset$). Moreover, in \cref{lem:reduction}, we prove that if $\widetilde{G}$ satisfies \cref{conj:spanning} with a tree decomposition of width $k$, then $G$ satisfies \cref{conj:strong} with a tree decomposition of width $k+1$.

The attached trees are chosen such that no two have \say{comparable} tree decompositions. More formally, given two trees $T_1$ and $T_2$, there is no tree decomposition $(T'_2, \B'_2)$ of $T_2$ of width $k$ such that $\T'_2$ is a subgraph of $T_1$, and likewise with the roles of $T_1$ and $T_2$ reversed. We do not frame the argument in this way, but it is the underlying reason our proof works. We accomplish this condition by, up to symmetry between $T_1$ and $T_2$, making height of $T_2$ much larger than the height of $T_1$, and the width of $T_1$ much larger than $\abs{V(T_2)}$. See \cref{fig:construction} for a depiction. 

With this intuition, we are ready to state the main definition.

\begin{definition}\label{dfn:construction}
Fix a positive integer $k$, a graph $G$, and an arbitrary ordering $a_1, \dotsc, a_n$ of the vertices of $G$. Then let $\widetilde{G}$ be the graph which is constructed from $G$ as follows.
\begin{itemize}
\item First define integers $2 = h_1 \ll h_2 \ll \dotsb \ll h_n$ as follows. Given $h_{j-1}$, we define $h_{j}\coloneqq (k+2)^{2h_{j-1}}+1$. Thus, by \cref{lem:longpath}, if $P$ is a path on at least $h_{j}-1$ vertices and $(T, \B)$ is a tree decomposition of $P$ of width at most $k$, then $T$ contains a path of length $2h_{j-1}$.
\item Next define integers $(k + 1)n +1= w_n \ll w_{n - 1} \ll \dotsb \ll w_1$ and corresponding rooted trees $S_n, S_{n-1}, \dotsc, S_1$ as follows. Given $w_j$, define $S_j$ to be the complete rooted $w_j$-ary tree of height $h_j$. Then, given $w_n, w_{n-1}, \dotsc, w_{j+1}$ and $S_n, S_{n-1}, \dotsc, S_{j+1}$, define
\begin{equation*}
    w_j \coloneqq (k + 1)\biggl(n + \sum_{i=j+1}^n \abs{V(S_i)}\biggr)+1.
\end{equation*}
\end{itemize}
Finally, let $\widetilde{G}$ be the graph which is obtained from the disjoint union of $G, S_1, S_2, \dotsc, S_n$ by, for each $j \in \set{1,2,\dots, n}$, identifying $a_j$ with the root of $S_j$. 
\end{definition}

\tikzstyle{labelNode}= [circle, draw, inner sep=2,outer sep=0, fill=tikzgray, minimum size=5pt]
\tikzstyle{triangle}=[fill=darkblue!30!white, draw=darkblue!50!white]
\begin{figure}[ht]
    \centering
    \begin{tikzpicture}
        \usetikzlibrary{math}
        \tikzmath{\x1 = 0; \x2 = 4; \x3 = 10;
            \h1 = 1.75; \h2 = 3; \h3 = 6;
            \w1=4; \w2=2; \w3=1;} 
        \foreach \i in {1,2,3} {   
        \node[labelNode] (a\i) at (\x\i,0) {};
        \draw[triangle] (a\i) -- (\x\i - 2, -\h\i) -- (\x\i + 2,-\h\i) -- (a\i);
        \draw[<->,latex-latex,draw=darkblue, thick] ( \x\i +1.5, 0)-- ( \x\i +1.5, -\h\i) node[midway] (h\i) {};
        \node at (\x\i,-\h\i -0.5) {\ifthenelse{\i=3}{$S_n$}{ $S_\i$}};
        \foreach \j in {0,...,\w\i} {
            \draw (\x\i - \j/4,-1) -- (a\i);
            \draw (\x\i + \j/4,-1) -- (a\i);
        }
        \draw[<->,latex-latex,draw=darkblue, thick] ( \x\i- 1/\i, -1.2) -- ( \x\i+ 1/\i, -1.2) node[midway,below] {\ifthenelse{\i=3}{$w_n$}{$w_\i$}};
       }
       \node[labelNode, label = above:{$a_1$}] at (0, 0) {};
       \node[labelNode, label = above:{$a_2$}] at (\x2, 0) {};
       \node[labelNode, label = above:{$a_n$}] at (\x3, 0) {};
        \node[label = right:{$h_1$}] at (1.3, -1) {};
        \node[label = right:{$h_2$}] at (h2) {};
        \node[label = right:{$h_n$}] at (h3) {};

    \begin{pgfonlayer}{background}
        \draw[draw=tikzgray, ultra thick] (\x3/2+\x1/2,0) ellipse (\x3/1.5 and .9);
        \node at (-1,1) {\textcolor{tikzgray}{\Large $G$}};
        \node at (\x3/2+\x2/2, 0) {\large \dots};
    \end{pgfonlayer}
    \end{tikzpicture}
    \caption{The graph $\widetilde{G}$ which is obtained from $G$ by attaching the complete $w_j$-ary tree $S_j$ of height $h_j$ to each vertex $a_j\in V(G)$.}
    \label{fig:construction}
\end{figure}
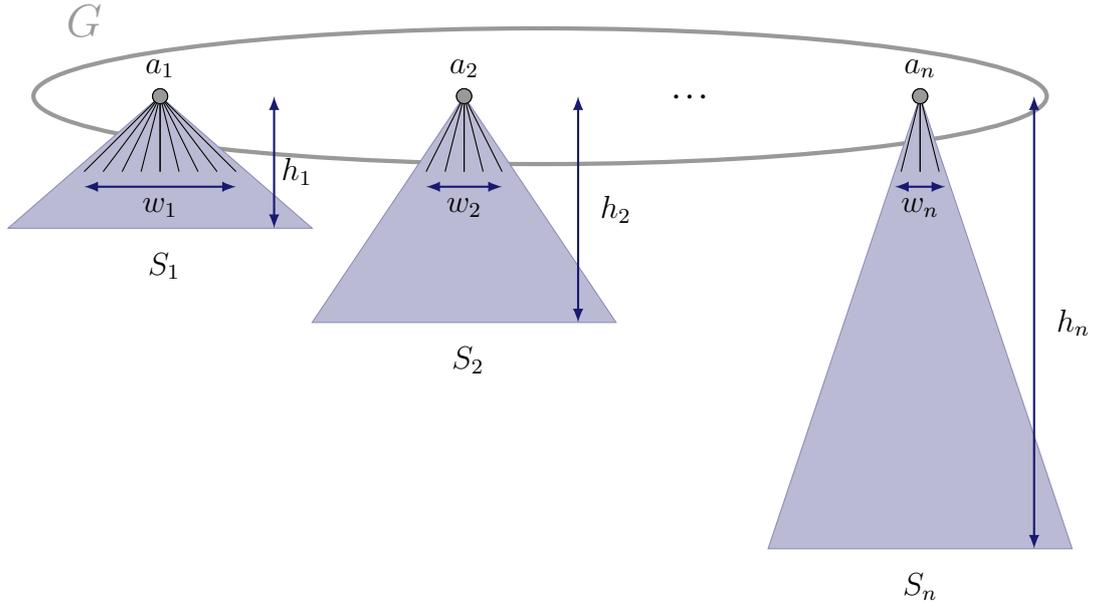

Note that the graph $\widetilde{G}$ from \cref{dfn:construction} can be obtained from $G$ by adding pendant vertices one at a time. It follows that $\tw(\widetilde{G}) = \max\left(\tw(G), 1 \right)$. The next key lemma therefore completes the reduction from \cref{conj:spanning} to \cref{conj:strong}.

\begin{lemma}\label{lem:reduction}
Let $k$ be a positive integer, let $G$ be a connected graph, let $a_1, \dotsc, a_n$ be an ordering of the vertices of $G$, and let $\widetilde{G}$ be the resulting graph constructed using \cref{dfn:construction}. Suppose that $\widetilde{G}$ has a tree decomposition $(T, \B)$ of width at most $k$ such that $T$ is a spanning tree of $\widetilde{G}$. 

Then there exists a tree decomposition $(T', \B')$ of $G$ of width at most $k+1$ such that $T'$ is a spanning tree of $G$ and for every $v \in V(G)$, we have $v \in T'_v$.
\end{lemma}

\begin{proof}
We use the notation introduced in \cref{dfn:construction}, except that we view each tree $S_j$ as an induced subgraph of $\widetilde{G}$ which is rooted at the vertex $a_j \in V(G)$. For the sake of convenience, we do not distinguish between $S_j$ and its vertex set.

In the first few claims we deduce roughly where each subtree $T_v$ (as defined in \cref{def:td-subtree-version}) lies. We say that two sets \emph{meet} if their intersection is non-empty.

\begin{claim}\label{claim:left}
    For every $j \in \set{1,2, \dotsc, n}$ and every non-leaf vertex $v$ of $S_j$, the set $T_v$ meets $(S_1 \cup \dotsb \cup S_j) \setminus V(G)$.
\end{claim}

\begin{proof}
    Consider the union of the bags of $(T, \B)$ that contain $v$. Each bag has size at most $k + 1$, so this union has size at most $(k + 1) \abs{V(T_v)}$. On the other hand, this union contains every neighbour of $v$ in $\widetilde{G}$ and so, by the choice of $w_j$,
    \begin{equation*}
        (k + 1) \abs{V(T_v)} \geq \deg_{\widetilde{G}}(v) \geq \deg_{S_j}(v) \geq w_j > (k + 1) \biggl(\abs{V(G)} + \sum_{i = j + 1}^{n} \abs{V(S_i)}\biggr).
    \end{equation*}
    In particular, $T_v$ is not a subgraph of $S_{j + 1} \cup \dotsb \cup S_n \cup G$. The claim follows.
\end{proof}

We say that a vertex $v$ of $\widetilde{G}$ is \emph{free} if $T_v$ meets $V(G)$ or, equivalently, if $v \in B_{a_1} \cup \dotsb \cup B_{a_n}$.
Otherwise, we call $v$ \emph{constrained}. Note that if $v$ is constrained, then $T_v$ is a subgraph of some $S_j - a_j$ since $T_v$ is a subtree of $\widetilde{G}$. The number of free vertices is at most
\begin{equation*}
    \abs{B_{a_1} \cup \dotsb \cup B_{a_n}} \leq (k + 1)n,
\end{equation*}
and so almost all vertices are constrained.

\begin{claim}\label{claim:children}
    For every $j \in \set{1,2, \dotsc, n}$, the vertex $a_j$ has a child $b_j$ in $S_j$ such that $T_{b_j}$ is a subgraph of $S_j - a_j$.
\end{claim}

\begin{proof}
    As $S_j$ is a complete $w_j$-ary tree of height $h_j$, there are $w_j$ vertex-disjoint paths that start at the children of $a_j$ and end at parents of leaves of $S_j$. Since $w_j > (k + 1)n$, at least one of these paths contains no free vertices -- call this path $P$. Let $T_P = \cup_{v \in V(P)} T_v$. So $T_P$ is a subtree of $\widetilde{G}$. Each $v \in V(P)$ is constrained, and so for each $v$ there is an $i$ such that $T_v$ is a subgraph of $S_i - a_i$. But, as $T_P$ is connected and there is no edge between different $S_i - a_i$, this $i$ must be the same for all $v \in V(P)$. That is, there is some $i$ such that $T_P$ is a subgraph of $S_i - a_i$. Let $b_j$ be the child of $a_j$ that is a vertex of $P$ (since $h_1 \geq 2$, such a $b_j$ exists).

    By \cref{claim:left}, we have that $T_{b_j}$ meets $(S_1 - a_1) \cup \dotsb \cup (S_j - a_j)$. Since $T_{b_j}$ is a subgraph of $T_P$, we have $i \leq j$. Next focus on the tree decomposition $(T_P, \B_P)$ of $P$ where $\B_P$ is $\B$ restricted to the vertices of $P$. This tree decomposition has width at most $k$. The path $P$ contains $h_j - 1$ vertices and so, by the choice of $h_j$, the tree $T_P$ must contain a path of length at least $2h_{j - 1}$. However, $T_P$ is a subgraph of $S_i - a_i$ whose longest paths have length less than $2h_i$. In particular, this implies that $h_i > h_{j - 1}$ and so $i \geq j$. Thus $i = j$ and $b_j$ is as required.
\end{proof}

We say that a vertex $a_j \in V(G)$ is \emph{grounded} if $T_{a_j}$ contains $a_j$.

\begin{claim}\label{claim:grounded}
    If a vertex $a_j \in V(G)$ is not grounded, then $T_{a_j}$ is a subgraph of $S_j - a_j$ and every neighbour $a_i \in V(G)$ of $a_j$ is both grounded and satisfies ${a_j} \in T_{a_i}$.
\end{claim}

\begin{proof}
    Suppose that $a_j \in V(G)$ is not grounded. Then $a_j \notin T_{a_j}$. Let $b_j$ be the child of $a_j$ given by \cref{claim:children}. As $a_j$ and $b_j$ are adjacent, $T_{a_j}$ meets $T_{b_j}$. But $T_{b_j}$ is a subgraph of $S_j - a_j$, and so $T_{a_j}$ meets $S_j - a_j$. However, $T_{a_j}$ is connected and does not contain $a_j$, so $T_{a_j}$ must be a subgraph of $S_j - a_j$.

    Let $a_i \in V(G)$ be a neighbour of $a_j$. If $a_i$ is not grounded, then $T_{a_i}$ is a subgraph of $S_i - a_i$. But then $T_{a_i}$ and $T_{a_j}$ do not meet, which is impossible as $a_i$ and $a_j$ are adjacent. Thus $a_i$ is grounded. Now $T_{a_i}$ and $T_{a_j}$ meet and so $T_{a_i}$ meets $S_j - a_j$. So, since $T_{a_i}$ is connected, $T_{a_i}$ contains $a_j$. 
\end{proof}

We now define a tree decomposition of $G$ which satisfies \cref{lem:reduction}. First, let $T'$ be the subgraph of $T$ induced by $V(G)$; notice that $T'$ is a spanning tree of $G$ since $T$ is a spanning tree of $\widetilde{G}$. Next, delete all bags $B_x$ where $x \notin V(T')$ and delete all vertices of $\widetilde{G}$ that are not vertices of $G$. Finally, if a vertex $a_j$ is not grounded, then add $a_j$ to the bag $B_{a_j}$. Call the resulting collection of bags $\B' = (B'_{a_j})_{1 \leq j \leq n}$. We claim that $(T', \B')$ is a tree decomposition of $G$. This completes the proof of \cref{lem:reduction} since it is clear that $(T', \B')$ has width at most $k + 1$, that $T'$ is a spanning tree of $G$, and that for every $a_j \in V(G)$, we have $a_j \in T'_{a_j}$. 

Notice that if a vertex $a_j \in V(G)$ is grounded, then $T'_{a_j}$ is just the induced subgraph of $T_{a_j}$ restricted to $V(G)$; so $T'_{a_j}$ is still connected. Likewise, if $a_j$ is not grounded, then by \cref{claim:grounded}, $T'_{a_j} = \set{a_j}$ is connected. We are left to check that for every edge $a_i a_j \in E(G)$, the subtrees $T'_{a_i}$ and $T'_{a_j}$ meet. First suppose that $a_j$ is not grounded. Then, by \cref{claim:grounded}, $T'_{a_i}$ and $T'_{a_j}$ both contain the vertex $a_j$. The case that $a_i$ is not grounded is symmetric, so we may assume that both $a_i$ and $a_j$ are grounded. As $a_i$ and $a_j$ are adjacent in $\widetilde{G}$, the trees $T_{a_i}$ and $T_{a_j}$ meet in $\widetilde{G}$. Let $\ell \in \set{1,2,\dotsc, n}$ be such that they meet in $S_{\ell}$. Now $T_{a_i}$ contains $a_i$ and is connected and $T_{a_j}$ contains $a_j$ and is connected, so $T_{a_i}$ and $T_{a_j}$ both contain $a_{\ell}$. So both $T'_{a_i}$ and $T'_{a_j}$ contain $a_{\ell}$, as required. This completes the proof of \cref{lem:reduction}.
\end{proof}

\section{Construction}\label{sec:construction}
In this section we disprove \cref{conj:strong} and then combine the previous reductions in order to prove \cref{thm:main}. 

We begin by defining the relevant graphs. Then we prove that they are counterexamples in Lemmas~\ref{lem:refl-tree-intersect} and~\ref{lem:construction}.

\begin{definition}
    The \emph{first reflected-tree}, which we denote by $G_1$, is the singleton graph with exactly one vertex and no edges. We call the vertex of $G_1$ its \emph{root} vertex. Then, for any positive integer $r \geq 2$, the \emph{$r$-th reflected-tree} $G_r$ is constructed recursively as follows:
    \begin{itemize}

        \item Let $H$ and $H'$ be two disjoint copies of $G_{r-1}$, and let $u$ and $v$ be two new vertices, which we call the \emph{root} vertices of $G_r$. To construct $G_r$, we start with $H$ and $H'$, then make $u$ adjacent to a root vertex of $H$ and a root vertex of $H'$. Finally, we make $v$ adjacent to the remaining root vertex of $H$ and the remaining root vertex of $H'$. See \cref{reflected-trees} for a depiction.
    \end{itemize}
\end{definition}

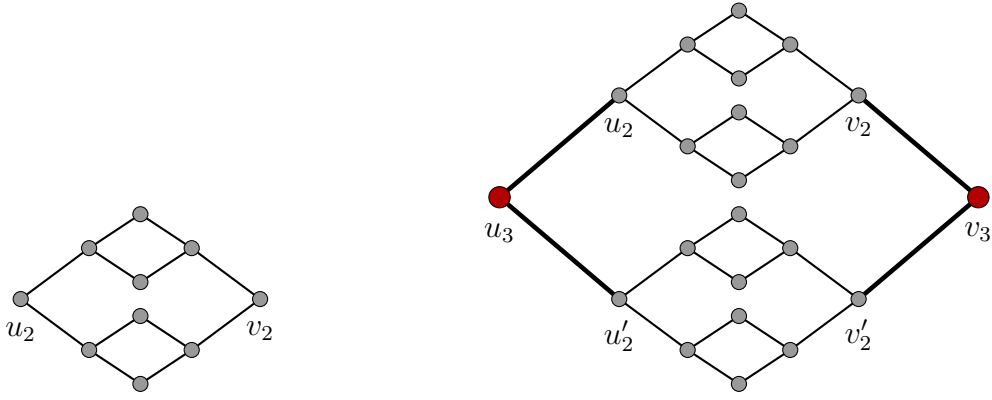
\begin{figure}
\centering
\begin{tikzpicture}[scale = .9, every node/.style={MyNode}]
    \def \r {2}
    \def \start {90}
    \def \cnt {0.75} 
    \node[label=below:{$u_2$},draw=black](X2) at (-1.75,0) {};
    \node[label=below:{$v_2$},draw=black](Z2) at (1.75,0) {}; 
    
    \node[draw=black] (X0) at (-0.75,0+\cnt) {} ;
    \node[draw=black](Y0) at (0, 0.5+\cnt) {};
    \node[draw=black](Y1) at (0, -0.5+\cnt) {};
    \node[draw=black](Z0) at (0.75,0+\cnt) {} ;

    \node[draw=black](X1) at (-0.75,0-\cnt) {} ;
    \node[draw=black](Y2) at (0, 0.5-\cnt) {};
    \node[draw=black](Y3) at (0, -0.5-\cnt) {};
    \node[draw=black](Z1) at (0.75,0-\cnt) {} ;

    \draw[thick] (X0) -- (Y0);
    \draw[thick] (X0) -- (Y1);
    \draw[thick] (X0) -- (X2);
    
    \draw[thick] (X1) -- (Y2);
    \draw[thick] (X1) -- (Y3);
    \draw[thick] (X1) -- (X2);
    
    \draw[thick] (Z0) -- (Y0);
    \draw[thick] (Z0) -- (Y1);
    \draw[thick] (Z0) -- (Z2);

    \draw[thick] (Z1) -- (Y2);
    \draw[thick] (Z1) -- (Y3);
    \draw[thick] (Z1) -- (Z2);
\end{tikzpicture}\hskip2.5cm
\begin{tikzpicture}[scale = .9, every node/.style={MyNode}]
    \def \r {2}
    \def \start {90}
    \def \cnt {0.75} 
    \def \cnst {1.5} 
     \node[label=below:{$u_3$}, fill=tikzred, minimum size=8pt] (X3) at (-3.5,0) {};
     \node[label=below:{$v_3$}, fill=tikzred, minimum size=8pt] (Z3) at (3.5,0) {};

    \node[label=below:{$u_2$},draw=black](X2) at (-1.75,0+\cnst) {};
    \node[label=below:{$v_2$},draw=black](Z2) at (1.75,0+\cnst) {}; 
    \node[draw=black](X0) at (-0.75,0+\cnt+\cnst) {} ;
    \node[draw=black](Y0) at (0, 0.5+\cnt+\cnst) {};
    \node[draw=black](Y1) at (0, -0.5+\cnt+\cnst) {};
    \node[draw=black](Z0) at (0.75,0+\cnt+\cnst) {} ;
    \node[draw=black](X1) at (-0.75,0-\cnt+\cnst) {} ;
    \node[draw=black](Y2) at (0, 0.5-\cnt+\cnst) {};
    \node[draw=black](Y3) at (0, -0.5-\cnt+\cnst) {};
    \node[draw=black](Z1) at (0.75,0-\cnt+\cnst) {} ;
    
    \node[label=below:{$u_2'$},draw=black](X21) at (-1.75,0-\cnst) {};
    \node[label=below:{$v_2'$},draw=black](Z21) at (1.75,0-\cnst) {}; 
    \node[draw=black](X01) at (-0.75,0+\cnt-\cnst) {} ;
    \node[draw=black](Y01) at (0, 0.5+\cnt-\cnst) {};
    \node[draw=black](Y11) at (0, -0.5+\cnt-\cnst) {};
    \node[draw=black](Z01) at (0.75,0+\cnt-\cnst) {} ;
    \node[draw=black](X11) at (-0.75,0-\cnt-\cnst) {} ;
    \node[draw=black](Y21) at (0, 0.5-\cnt-\cnst) {};
    \node[draw=black](Y31) at (0, -0.5-\cnt-\cnst) {};
    \node[draw=black](Z11) at (0.75,0-\cnt-\cnst) {} ;

    \draw[thick] (X0) -- (Y0);
    \draw[thick] (X0) -- (Y1);
    \draw[thick] (X0) -- (X2);
    \draw[thick] (X1) -- (Y2);
    \draw[thick] (X1) -- (Y3);
    \draw[thick] (X1) -- (X2);
    \draw[thick] (Z0) -- (Y0);
    \draw[thick] (Z0) -- (Y1);
    \draw[thick] (Z0) -- (Z2);
    \draw[thick] (Z1) -- (Y2);
    \draw[thick] (Z1) -- (Y3);
    \draw[thick] (Z1) -- (Z2);

    \draw[thick] (X01) -- (Y01);
    \draw[thick] (X01) -- (Y11);
    \draw[thick] (X01) -- (X21);
    \draw[thick] (X11) -- (Y21);
    \draw[thick] (X11) -- (Y31);
    \draw[thick] (X11) -- (X21);
    \draw[thick] (Z01) -- (Y01);
    \draw[thick] (Z01) -- (Y11);
    \draw[thick] (Z01) -- (Z21);
    \draw[thick] (Z11) -- (Y21);
    \draw[thick] (Z11) -- (Y31);
    \draw[thick] (Z11) -- (Z21);

    \draw[draw = black,line width=1.75pt] (Z3) -- (Z2);
    \draw[draw = black,line width=1.75pt] (Z3) -- (Z21);
    \draw[draw = black, line width=1.75pt] (X3) -- (X2);
    \draw[draw = black, line width=1.75pt] (X3) -- (X21);
    
\end{tikzpicture}
\caption{The $4$th reflected-tree $G_4$ (right, with root vertices larger and in red) being constructed from the $3$rd reflected-tree $G_3$ (left).}
\label{reflected-trees}
\end{figure}

Now we prove a lemma about the spanning trees of the reflected-tree. Whenever $T$ is a spanning tree of a graph $G$, we denote the fundamental cycle of an edge $e\in E(G)\setminus E(T)$ with respect to $T$ by $C^e_T$; thus $C^e_T$ is the unique cycle in the graph obtained from $T$ by adding $e$. 
\begin{lemma}\label{lem:refl-tree-intersect}
    For any integer $r\ge 2$ and any spanning tree $T$ of $G_r$, there is a matching $M\subseteq E(G_r)\setminus E(T)$ of size $r-1$ such that
    $$\bigcap_{e\in M} V\left(C^e_T\right) \neq \varnothing.$$
\end{lemma}
\begin{proof}
    Let $u$ and $v$ be the root vertices of $G_r$, and denote the path between them in $T$ by $P_{uv}$. Under the same conditions, we prove the following stronger outcome holds by induction:
    $$\bigcap_{e\in M} E\left(C^e_T\right) \cap E(P_{uv}) \neq \varnothing,$$
    for some matching $M\subseteq E(G_r)\setminus E(T)$ of size $r-1$.
    For the base case of $r=2$, the graph $G_r$ is a cycle on four vertices; then any spanning tree $T$ of $G_2$ is a path on four vertices, and we can take $M$ to be the $1$-edge matching $E(G_2)\setminus E(T)$.  
    
    Next, we may assume that $r > 2$ and the claim holds for $r-1$. 
    By definition, $G _r - \{u, v\}$ has exactly two connected components both of which are isomorphic to $G_{r-1}$.
    We denote these components by $H$ and $H'$. Exactly one of $T_H\coloneqq T[V(H)\cup\{u,v\}]$ and $T_{H'}\coloneqq T[V(H')\cup\{u,v\}]$ is connected in $T$; without loss of generality, we assume that $T_H$ is connected in $T$. We can apply the inductive hypothesis on $T[V(H)]$, which is a spanning tree of $H$, to find a matching $M_H\subseteq E(H)\setminus E(T[V(H)])$ of size $r-2$ with 
    $\bigcap_{e\in M_H} E\left(C^e_T\right) \cap E(P_{uv}-\{u,v\}) \neq \varnothing$. The other subgraph $T_{H'}$ is not connected. In fact, it contains exactly two components: one containing $u$, and the other containing $v$. Thus there exists an edge $e' \in E(G_r)\setminus E(T)$ with one end in each of these two components of $T_{H'}$. Observe that $e'$ lies in $G[V(H')\cup\{u,v\}]$ which is vertex-disjoint from $H$. Thus $M_H \cup \{e'\}$ is a matching since $M_H \subseteq H$.

    For convenience, let us define $M\coloneqq M_H \cup \{e'\}$. $M$ is a matching of size $r-1$, and we have $E(P_{uv}) \subseteq E(C^{e'}_T)$. From here, it follows that 
    $$\bigcap_{e\in M} E\left(C^e_T\right) \cap E(P_{uv}) \neq \varnothing.$$

    Thus $M$ is our desired matching.
\end{proof}

We are now ready to prove the following lemma, which shows that reflected-trees are a counterexample to \cref{conj:strong}.

\begin{lemma}\label{lem:construction}
For every $k \in \mathbb{N}$, if $(T, \B)$ is a tree decomposition of $G_{k+2}$ such that $T$ is a spanning tree of $G_{k+2}$ and, for every $v \in V(G_{k+2})$, we have $v \in T_v$, then the width of $(T, \B)$ is at least~$k$.
\end{lemma}
\begin{proof}
We begin by finding a matching $M\coloneqq \{u_1v_1,\ldots, u_{k+1}v_{k+1}\}\subseteq E(G_{k+2})\setminus E(T)$ of size $k+1$ satisfying the properties in \Cref{lem:refl-tree-intersect}. Let $x\in \bigcap_{e\in M} V\left(C^e_T\right)$. By construction, $x$ is in the path $P_{u_iv_i}$ between $u_i$ and $v_i$ in $T$ for every $i\in\{1,\ldots,k+1\}$. From \Cref{def:td-subtree-version},  the trees $T_{u_i}$ and $T_{v_i}$ intersect; furthermore, since $T_{u_i}$ and $T_{v_i}$ are connected in $T$, with $u_i\in V(T_{u_i})$ and $v_i\in V(T_{v_i})$, we find that every vertex of $P_{u_iv_i}$ is in $V(T_{u_i}) \cup V(T_{v_i})$. As a result, $x\in V(T_{u_i})\cup V(T_{v_i})$. That is, $u_i\in B_x$ or $v_i\in B_x$ for all $i\in\{1,\ldots,k+1\}$. Since $M$ is a matching, we have that $|B_x| \ge k+1$ and the width of $(T,\mathcal{B})$ is at least $k$.
\end{proof}

We are now ready to prove the main theorem, which is restated below for convenience.

\mainThm*
\begin{proof}
For convenience, we fix an integer $k \geq 2$. Now consider the $(k+3)$-rd reflected-tree $G_{k+3}$. Let $a_1, \dotsc, a_n$ be an arbitrary ordering of $V(G_{k+3})$, and let $\widetilde{G}_{k+3}$ be the graph obtained from the integer $k-1$, the graph $G_{k+3}$, and the ordering $a_1, \dotsc, a_n$ by applying \cref{dfn:construction}. 

We now prove that $\widetilde{G}_{k+3}$ satisfies the conditions of \cref{thm:main}.
First, recall that $\widetilde{G}_{k+3}$ has treewidth equal to $\max(\tw(G_{k+3}), 1)$. Moreover, $\tw(G_{k+3})=2$ since $G_{k+3}$ is series parallel and not a tree. Thus $\widetilde{G}_{k+3}$ is a connected graph of treewidth~$2$, as desired.

Next, suppose towards a contradiction that $\widetilde{G}_{k+3}$ has a tree decomposition $(T, \B)$ of width at most $k-1$ such that $T$ is a minor of $\widetilde{G}_{k+3}$. Since $\widetilde{G}_{k+3}$ is connected, \cref{lem:minor-to-subgraph-tree} says that $\widetilde{G}_{k+3}$ has a tree decomposition $(T', \B')$ of width at most $k-1$ such that $T'$ is a spanning tree of $\widetilde{G}_{k+3}$. Thus, since $G_{k+3}$ is connected, \cref{lem:reduction} says that $G_{k+3}$ has a tree decomposition $(T'', \B'')$ of width at most $k$ such that $T''$ is a spanning tree of $G_{k+3}$ and for every $v \in V(G_{k+3})$, we have $v \in T''_v$. However, \cref{lem:construction} also says that $(T'', \B'')$ has width at least $k+1$. This contradiction completes the proof of \cref{thm:main}.
\end{proof}

\section*{Acknowledgements}
The authors would like to thank Sang-il Oum for his suggestions about where to look for a counterexample to \cref{conj:strong}; Zden\v{e}k Dvo\v{r}\'{a}k for sharing the original problem and helpful comments which improved the paper; David Wood for pointing out all of the progress in~\cite{hickingbothamThesis}, and Sophie Spirkl for her helpful feedback on our proof that \cref{conj:strong} is false. 
Aristotelis Chaniotis, Linda Cook, Sepehr Hajebi, and Sophie Spirkl asked about a counterexample to a strengthening of \cref{conj:strong} at the Barbados Graph Theory Workshop in December 2022, which started this whole work.
As such, the authors would like to thank Aristotelis Chaniotis, Sepehr Hajebi, Sophie Spirkl, and the organizers of the Barbados Graph Theory Workshop -- Sergey Norin, Paul Seymour, and David Wood.

{
\fontsize{11pt}{12pt}
\selectfont
	
\newcommand{\etalchar}[1]{$^{#1}$}

}
\end{document}